\documentclass[11pt,reqno]{amsart}

\usepackage{amsfonts, amsthm, amsmath}
\allowdisplaybreaks[4]

\usepackage{rotating}

\usepackage{tikz}

\usepackage{graphics}

\usepackage{amssymb}

\usepackage{amscd}

\usepackage[latin2]{inputenc}

\usepackage{t1enc}

\usepackage[mathscr]{eucal}

\usepackage{indentfirst}

\usepackage{graphicx}

\usepackage{graphics}

\usepackage{pict2e}

\usepackage{mathrsfs}

\usepackage{enumerate}
\usepackage{ferrers} 
\usepackage[pagebackref]{hyperref}
\hypersetup{colorlinks=true}
\usepackage{cite}
\usepackage{color}
\usepackage{epic}
\usepackage{hyperref} 
\usepackage{framed}
\usepackage{mathabx}

\numberwithin{equation}{section}
\theoremstyle{plain}

\newtheorem{theorem}{Theorem}[section]

\newtheorem{lemma}[theorem]{Lemma}

\newtheorem{corollary}[theorem]{Corollary}

\theoremstyle{definition}

\newtheorem{Def}[theorem]{Definition}

\newtheorem{example}[theorem]{Example}

\newtheorem{remark}[theorem]{Remark}

\newtheorem{?}[theorem]{Problem}

\newcommand{\abs}[1]{\left|#1\right|}
\newcommand{\set}[1]{\left\{#1\right\}}

\def\mod{\mathrm{mod}}
\def\A{\mathcal{A}}
\def\E{\mathcal{E}}
\def\sol{\mathrm{sol}}
\def\sl{\mathrm{sl}}
\def\la{\lambda}
\def\C{\mathcal{C}}
\def\N{\mathbb{N}}
\def\O{\mathcal{O}}
\def\T{\mathcal{T}}
\def\sfb{\mathrm{sfb}}
\def\gam{\gamma}
\newcommand{\B}{\mathcal{B}}
\newcommand{\D}{\mathcal{D}}

\renewcommand{\P}{\mathcal{P}}
\newcommand{\W}{\mathcal{W}}
\def\RR{\mathcal{RR}}
\def\lrp{\mathrm{lrp}}
\def\sle{\mathrm{sle}}

\begin{document}

\title[Sequences of odd length in strict partitions I]{Sequences of odd length in strict partitions I: the combinatorics of double sum Rogers-Ramanujan type identities}

\author[S. Fu]{Shishuo Fu}
\address[Shishuo Fu]{College of Mathematics and Statistics, Chongqing University, Chongqing 401331, P.R. China}
\email{fsshuo@cqu.edu.cn}

\author[H. Li]{Haijun Li}
\address[Haijun Li]{College of Mathematics and Statistics, Chongqing University, Chongqing 401331, P.R. China}
\email{lihaijun@cqu.edu.cn}

\date{\today}

\begin{abstract}
Strict partitions are enumerated with respect to the weight, the number of parts, and the number of sequences of odd length. We write this trivariate generating function as a double sum $q$-series. Equipped with such a combinatorial set-up, we investigate a handful of double sum identities appeared in recent works of Cao-Wang, Wang-Wang, Wei-Yu-Ruan, Andrews-Uncu, Chern, and Wang, finding partition theoretical interpretations to all of these identities, and in most cases supplying Franklin-type involutive proofs. This approach dates back more than a century to P.~A.~MacMahon's interpretations of the celebrated Rogers-Ramanujan identities, and has been further developed by Kur\c{s}ung\"{o}z in the last decade.
\end{abstract}


\maketitle

\section{Introduction}

First proven by Rogers in 1894 and then rediscovered by Ramanujan sometime before 1913, the Rogers-Ramanujan identities~\cite[Chap.~19]{HW08} state that for $\abs{q}<1$ we have 
\begin{align}\label{id:RR1}
\sum_{n=0}^{\infty}\frac{q^{n^2}}{(q;q)_n} &=\frac{1}{(q;q^5)_{\infty}(q^4;q^5)_{\infty}},\\
\label{id:RR2}
\sum_{n=0}^{\infty}\frac{q^{n^2+n}}{(q;q)_n} &=\frac{1}{(q^2;q^5)_{\infty}(q^3;q^5)_{\infty}}.
\end{align}
Here and in what follows, we adopt the following customary $q$-Pochhammer symbols and mostly follow the notations from Andrews' book~\cite{andtp}. For $\abs{q}<1$, we let
\begin{align*}
& (a;q)_n=(1-a)(1-aq)\cdots(1-aq^{n-1}), \text{ for $n\ge 1$,}\\
& (a;q)_0=1,\text{ and } (a;q)_{\infty}=\lim_{n\to\infty}(a;q)_n.
\end{align*}

The two formulas \eqref{id:RR1} and \eqref{id:RR2} were communicated to MacMahon, who stated them (without proofs) in his magnum opus {\it Combinatory Analysis}~\cite[Vol.~I, Sect.~VII, Ch.~III]{mac15}, in terms of the equinumerosities between two types of restricted partitions.
\begin{theorem}[Rogers-Ramanujan-MacMahon]\label{thm:RR-partition}
For each $n\ge 1$,
\begin{enumerate}
  \item there are as many partitions of $n$ into parts that are mutually at least $2$ apart, as partitions of $n$ into parts that are congruent to $\pm 1$ moduluo $5$;
  \item there are as many partitions of $n$ into parts greater than $1$ that are mutually at least $2$ apart, as partitions of $n$ into parts that are congruent to $\pm 2$ moduluo $5$.
\end{enumerate}
\end{theorem}

Proclaimed by Hardy~\cite[p.~28]{har37} as ``most remarkable'', this pair of identities \eqref{id:RR1}, \eqref{id:RR2}, and their combinatorial counterpart Theorem~\ref{thm:RR-partition} have spawn an enormous amount of work for over a century, spreading all over the subjects such as representation theory, quantum physics, etc. We direct the reader to Sills' book~\cite{sil18} for references and further information. 


It is worth noting that while it is straightforward to see that the right hand side of \eqref{id:RR1} is the generating function for partitions into parts congruent to $1$ or $4$ moduluo $5$ as described in Theorem~\ref{thm:RR-partition} (1), it does require some further explanation to see that the other set of restricted partitions (the one with the ``gap condition'') mentioned in Theorem~\ref{thm:RR-partition} (1) is indeed generated by the left hand side of \eqref{id:RR1}. We shall complete this task in the next section via the ``base $+$ increments'' framework (see remark~\ref{rmk:Kagan's credit} and the discussion before it) and treat it as the prototype of most combinatorial constructions that show up later. In general, it is often the case that $q$-series identities resembling \eqref{id:RR1} and \eqref{id:RR2} were derived first via assorted methods without any partition theoretical interpretations, and only later such interpretations were supplied and direct bijective proofs were constructed in some of the cases. In the present work, we aim to initiate such combinatorial investigations on the following three Rogers-Ramanujan type identities, previously established by Cao-Wang~\cite{CW23} and Wang-Wang~\cite{WW23} via integral method.
\begin{theorem}\label{thm:three (1,2)}
(Cf. \cite[Th.~3.8]{CW23} and \cite[Th.~3.1]{WW23}) We have
\begin{align}
\sum\limits_{i, j\geq 0}\frac{(-1)^{j}u^{i+j}q^{i^2+2ij+2j^2}}{(q;q)_{i}(q^2;q^2)_{j}} &=(-uq; q^2)_{\infty},\label{id:(1,2)-1}\\
\sum_{i, j\geq 0}\frac{u^{i+2j}q^{i^2+2ij+2j^2+j}}{(q; q)_{i}(q^2; q^2)_{j}} &=(-uq; q)_{\infty},\label{id:(1,2)-2}\\
\sum_{i, j\geq 0}\frac{u^{i+2j}q^{2i^2+4ij+4j^2-3i}}{(q^2; q^2)_{i}(q^4; q^4)_{j}} &=(1+uq+uq^{-1})(-uq^3; q^2)_{\infty}.\label{id:(1,2)-3}
\end{align}
\end{theorem}

Our combinatorial approach to interprete and derive the above three identities centers around the number of sequences of odd length in strict partitions, a notion that we believe is systematicly studied here for the first time. We begin by recalling some basic notations in the theory of integer partitions.

For a given non-negative integer $n$, a {\it partition} $\lambda$ of $n$ is a weakly increasing list\footnote{The more common convention is to require such a list to be weakly decreasing. We define it this way to facilitate the description of later operations on partitions.} of positive integers that sum up to $n$. We write $\lambda=\lambda_{1}+\lambda_{2}+\cdots+\lambda_{m}$ with $\lambda_{1}\leq \lambda_{2}\cdots\leq\lambda_{m}$, where the {\it weight} of $\lambda$ will be denoted by $|\lambda|=n$ and each $\lambda_{i}$ is called a {\it part} of $\lambda$ for $1\leq i\leq m$. The number of parts $m$ is called the {\it length} of the partition $\lambda$ and is denoted by $\ell(\lambda)$. We also require two variants of length. Let $\ell_d(\la)$ be the number of different part sizes that occur in $\la$, and let $\ell_r(\la)$ be the number of repeated parts in $\la$. Let $f_{i}$ be the number of times $i$ appears as a part in $\lambda$, for $1\leq i\leq n$. Denote by $\P(n)$ the set of all partitions of $n$, while its cardinality $|\P(n)|$ is denoted as $p(n)$ and we set $\P:=\bigcup_{n\ge 0}\P(n)$. For a given partition $\lambda$, its {\it Ferrers diagram}~\cite[p.~7]{andtp} is a graphical representation, denoted as $[\lambda]$, using left-justified rows of unit cells, such that the $i$-th row (from bottom up) consists of $i$ cells. For example, the Ferrers diagram of $2+4+7+7+9$ is shown in Figure~\ref{fig:Ferrers}. 
\begin{figure}[h!]
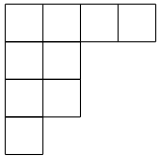

\begin{ferrers}
\addcellrows{4+2+2+1}
\end{ferrers}
\caption{The Ferrers diagram $[\lambda]$ for $\lambda=1+2+2+4$.}
\label{fig:Ferrers}
\end{figure}

Strict partitions refer to those who do not allow repetitions among the parts, or alternatively, those satisfying $f_i\le 1$ for all $i$. We denote the set of strict partitions by $\D$, and similarly $\D(n)$ stands for the subset wherein all partitions are of weight $n$. The following is the most important definition of this paper.

\begin{Def}
Given a strict partition $\la\in\D$, a maximal string of consecutive parts contained in $\la$ is called a {\it sequence} of $\la$. We denote the number of sequences in $\la$ that are of odd length by $\sol(\la)$. We define the generating function of strict partitions with respect to the weight, the length, and the number of sequences of odd length as
\begin{align*}
D^{\sol,\ell}(x,y;q) &:=\sum_{\la\in\D}x^{\sol(\la)}y^{\ell(\la)}q^{\abs{\la}}\\
&=1+xyq+xyq^2+(xy+y^2)q^3+(xy+x^2y^2)q^4+\cdots.
\end{align*}
\end{Def}

Clearly, each strict partition splits into sequences in a unique fashion. For instance, the partition $\la=2+4+5+8+9+10$ breaks into three sequences: $(2)$, $(4,5)$, and $(8,9,10)$, among which only two have odd lengths. Thus we see $\sol(\la)=2$ and $\la$ contributes $x^2y^6q^{38}$ to the generating function $D^{\sol,\ell}(x,y,q)$. 

The study of sequences in partitions dates back to Sylvester~\cite[Th.~2.12]{syl82} and MacMahon~\cite[Vol.~II, Sect.~VII, Ch.~IV]{mac15}; see Andrews' paper \cite[Sect.~2]{and20} for a brief history on previous work involving the sequences in partitions. Our treatment of sequences in strict partitions with parity consideration proves to be crucial for understanding the three identities in Theorem~\ref{thm:three (1,2)} from a combinatorial perspective. The following double sum expression for $D^{\sol,\ell}(x,y,q)$ is one of the main results of this paper.

\begin{theorem}\label{thm:main}
We have 
\begin{align}\label{gf:l-sol}
D^{\sol,\ell}(x,y;q) &= \sum_{i,j\ge 0}\frac{x^iy^{i+2j}q^{i^2+2ij+2j^2+j}}{(q;q)_i(q^2;q^2)_j}.
\end{align}
\end{theorem}
Basing on the aforementioned ``base $+$ increments'' framework, we present in the next section a bijective proof of \eqref{gf:l-sol}. Consequently, the series sides of the three identities in Theorem~\ref{thm:three (1,2)} could then be identified as specializations of the double $q$-series appearing in \eqref{gf:l-sol}. We re-derive these three identities in section~\ref{sec:3id} via combinatorial arguments. In particular, the proof of \eqref{id:(1,2)-1} features a Franklin-type involution. After that, as further applications of our combinatorial viewpoint, we investigate three more identities concerning double sum $q$-series.

In a recent work by Wei-Yu-Ruan~\cite{WYR23}, the following parametrized generalization of both \eqref{id:(1,2)-1} and \eqref{id:(1,2)-2} was derived via the contour integral method.

\begin{theorem}\label{thm:Wei-2para}
(Cf. \cite[Theorem 1.1]{WYR23}) Let $x, y$ be complex numbers. Then
\begin{equation}
\sum\limits_{i, j\geq 0}\frac{q^{i^2+2ij+2j^2-i-j}}{(q; q)_{i}(q^2; q^2)_{j}}x^{i}y^{2j}=(y; q)_{\infty}\sum\limits_{j\geq 0}\frac{(-x/y; q)_{j}}{(q; q)_{j}(y; q)_{j}}q^{\binom{j}{2}}y^{j}.\label{id:Wei-2para}
\end{equation} 
\end{theorem}
Note that the left hand side of \eqref{id:Wei-2para} is essentially the right hand side of \eqref{gf:l-sol} upon obvious change of variables. We are able to interprete the right hand side of \eqref{id:Wei-2para} as the generating function for certain partition pairs, and to construct a killing involution on these pairs so as to explain \eqref{id:Wei-2para} combinatorially. These will be accomplished in section~\ref{sec:Wei}.

Meanwhile, the following two double sum $q$-series identities resemble those three found in Theorem~\ref{thm:three (1,2)}, with $(q^2;q^2)_j$ replaced by $(q^3;q^3)_j$ and the exponents of $q$ in the numerator adjusted accordingly. Andrews and Uncu \cite{AU23} derived \eqref{id:AU-1} via $q$-difference equations, while \eqref{id:AU-2} was raised as a conjecture in that same paper. Two independent proofs of \eqref{id:AU-2} can be found in subsequent works of Chern~\cite{che23} and Wang~\cite{wan23}, respectively.
 
\begin{theorem}(Cf. \cite[Th. 1.1 and Conj. 1.2]{AU23})\label{thm:AU}
We have
\begin{align}
\sum_{i, j \ge 0} \frac{(-1)^j q^{i^2+3ij+\frac{3j(3j+1)}{2}}}{(q;q)_i (q^3;q^3)_j} &=\frac{1}{(q;q^3)_{\infty}},\label{id:AU-1}\\
\sum_{i, j \ge 0} \frac{(-1)^j q^{\frac{3 j(3 j+1)}{2}+i^2+3 i j+i+j}}{(q ; q)_i(q^3 ; q^3)_j} &=\frac{1}{(q^2, q^3 ; q^6)_{\infty}}.\label{id:AU-2}
\end{align}
\end{theorem}

As we shall demonstrate in section~\ref{sec:Andrews-Uncu}, our combinatorial framework can be easily adapted to give partition theoretical interpretations to certain parametrized double series that specializes to the left hand sides of both \eqref{id:AU-1} and \eqref{id:AU-2}. The paper concludes with outlook for future works.


\section{A combinatorial framework and a bijective proof of Theorem~\ref{thm:main}}\label{sec:bij}
Recall that the {\it sum} of two partitions is understood as the usual partwise summation. I.e., for two partitions $\la$ and $\mu$, their sum $\la+\mu$ is taken to be the partition whose $i$-th part is given by $\la_i+\mu_i$. Notice that for convenience, we may want to append as many zeros to a partition as we see fitting.

In this section, we resume the discussion after Theorem~\ref{thm:RR-partition} to explain that the series side of \eqref{id:RR1}, namely $\sum_{n\ge 0}q^{n^2}/(q;q)_n$, is indeed the generating function for $\RR$, the set of partitions whose parts are mutually at least $2$ apart. In fact, the summand $q^{n^2}/(q;q)_n$ generates the subset 
$$\RR_n:=\set{\la\in\RR: \ell(\la)=n}.$$
To see this, we rewrite the numerator as $q^{n^2}=q^{1+3+\cdots+(2n-1)}$, and realize that the partition in $\RR_n$ with the smallest weight is precisely given by
$$\beta^{(n)}:=1+3+5+\cdots+(2n-1).$$
We shall refer to $\beta^{(n)}$ as the {\it base partition} of the set $\RR_n$. Now for each partition $\la\in\RR_n$, we associate with it a unique partition $\iota=\iota_1+\iota_2+\cdots+\iota_n$ such that $\iota_i=\la_i-(2i-1)$ for $1\le i\le n$, or equivalently $\la=\beta^{(n)}+\iota$. Note that $\iota$ may contain zeros as parts and it has at most $n$ non-zero parts, thus is clearly seen to be generated by $1/(q;q)_n$. One sees that for a fixed $n$, the correspondence $\la\mapsto \iota$ is actually a bijection, from $\RR_n$ to $\P_{\le n}:=\set{\mu\in\P: \ell(\mu)\le n}$. We call the auxiliary partition $\iota$ the {\it incremental partition} and view the relation $\la=\beta^{(n)}+\iota$ as the ``base $+$ increments'' decomposition. An example of this decomposition is illustrated in Fig.~\ref{fig:decomp} below, where the second Ferrers diagram has been $2$-indented to highlight the staircase base partition.

\begin{figure}[ht]
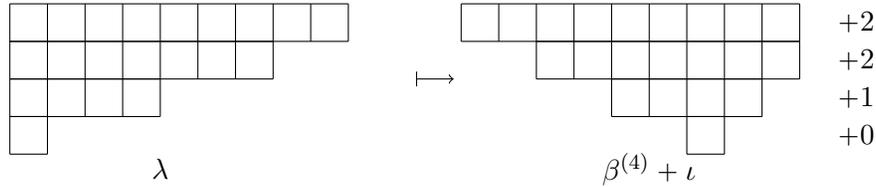

\begin{ferrers}
\addcellrows{9+7+4+1}
\addtext{2}{-2.2}{$\la$}
\transformto{}
\putright
\addcellrow[0]{9}
\addcellrow[2]{7}
\addcellrow[4]{4}
\addcellrow[6]{1}
\highlightcellbyletter{1}{11}{$+2$}
\highlightcellbyletter{2}{11}{$+2$}
\highlightcellbyletter{3}{11}{$+1$}
\highlightcellbyletter{4}{11}{$+0$}
\addtext{2.5}{-2.2}{$\beta^{(4)}+\iota$}
\end{ferrers}
\caption{The correspondence between $\lambda=1+4+7+9$ and $\iota=0+1+2+2$.}
\label{fig:decomp}
\end{figure}

Generally speaking, this ``base $+$ increments'' framework, with the above discussion being its prototype, can be described as follows. For a specific $q$-series (most likely double series in the sequel) that we would like to give a partition theoretical interpretation, we first identify its numerator as the weight for the unique base partition in a certain subset of restricted partitions. Then we examine the incremental partition(s) incurred by the denomenator and explain how to incorporate them into the base partition. Finally, we need to justify that the whole process is well-defined and reversible. In what follows, we simply call this process a b$+$i-decomposition.

\begin{remark}\label{rmk:Kagan's credit}
This whole ``base $+$ increments'' machinery aligns with MacMahon's interpretations of Rogers-Ramanujan identities given in Theorem~\ref{thm:RR-partition}, and has its genesis in the work of Kur\c{s}ung\"{o}z~\cite{kur10}, where a new combinatorial construction was given to the multiple series side of the famous Andrews-Gordon identity; see \cite{kur19,kur21,KS22} for several recent applications of this approach to derive evidently positive series for generating various kinds of restricted partitions.
\end{remark}

Now we proceed to outline the b$+$i-decomposition for strict partitions and establish Theorem~\ref{thm:main}. For $i,j\ge 0$, we define 
$$\D_{i,j}:=\set{\la\in\D: \sol(\la)=i, \text{ and }\ell(\la)=i+2j}.$$

The first step is to identify the base partition, say $\beta^{(i,j)}$, in $\D_{i,j}$. To minimize the weight, there should be no sequences of even length in $\beta^{(i,j)}$ since adjoining them with neighboring sequences could reduce the weight while preserving the statistics $\ell$ and $\sol$. For instance, both $1+2+4$ and $1+3+4$ should be turned into $1+2+3$. The following comparison for $b\ge 1$ and $c\ge 3$
\begin{align*}
&\underbrace{a+(a+1)+\cdots+(a+b+1)}_{b+2}+\underbrace{(a+b+3)+(a+b+4)+\cdots+(a+b+c)}_{c-2}\\
&\quad <\underbrace{a+(a+1)+\cdots+(a+b-1)}_{b}+\underbrace{(a+b+1)+(a+b+2)+\cdots+(a+b+c)}_{c}
\end{align*}
reveals that the $i$ odd-length sequences in $\beta^{(i,j)}$ are one having length $2j+1$ followed by $i-1$ singletons (i.e., sequences of length $1$). Namely, we have
\begin{align}\label{base for main thm}
\beta^{(i,j)}:=1+2+\cdots+(2j-1)+2j+(2j+1)+(2j+3)+\cdots+(2j+2i-1).
\end{align}
Note that $|\beta^{(i,j)}|=i^2+2ij+2j^2+j$, therefore as a strict partition itself, the base partition $\beta^{(i,j)}$ contributes to the generating function $D^{\sol,\ell}(x,y;q)$ precisely $x^{i}y^{i+2j}q^{i^2+2ij+2j^2+j}$, which matches the numerator of the summand in \eqref{gf:l-sol}.

Observing the denomenator $1/(q;q)_i(q^2;q^2)_j$, we expect the increments to come from two partitions. Let $\E$ denote the set of partitions into only even parts and similarly define the subset $\E_n:=\set{\la\in\E: \ell(\la)\le n}$. The discussion above indicates that Theorem~\ref{thm:main} is equivalent to the following lemma.

\begin{lemma}\label{lem:main bij}
For $i,j\ge 0$, there exists a bijection
\begin{align*}
\varphi=\varphi_{i,j}: \set{\beta^{(i,j)}}\times\P_i\times\E_j &\to \D_{i,j}\\
(\beta^{(i,j)},\mu,\eta) &\mapsto \la,
\end{align*}
such that $|\la|=|\beta^{(i,j)}|+|\mu|+|\eta|$, $\ell(\la)=\ell(\beta^{(i,j)})$, and $\sol(\la)=\sol(\beta^{(i,j)})$.
\end{lemma}

Before giving the proof, we take a closer look at strict partitions with the statistics $\sol$ and $\ell$ in mind, lay some groundwork and motivate our construction of $\varphi$. ``Sequence of odd length'' will be abbreviated as ``o-sequence'' in what follows. For any given strict partition $\la$, we call the largest part in each of $\la$'s o-sequence a {\it singleton}. All non-singleton parts can be grouped into {\it pairs} of consecutive parts. For simpler notation, we shall replace all plus signs with commas and place square brackets around each pair. So for instance the base partition can now be written as
\begin{align*}
\beta^{(i,j)} &= [1,2],[3,4],\ldots,[(2j-1),2j],(2j+1),(2j+3),\ldots,(2j+2i-1).
\end{align*}

We append copies of zeros to $\mu$ (resp.~$\eta$) if needed, so that $\mu$ (resp.~$\eta$) has exactly $i$ (resp.~$j$) parts. Now, going from the triple $(\beta^{(i,j)},\mu,\eta)$ to $\la$, the idea of $\varphi$ is to incorporate parts of $\mu$ and $\eta$ into the base partition $\beta^{(i,j)}$, one part at a time, by applying forward moves (to be given in the proof) to singletons and pairs, to get a string of strict partitions
$$\la^{(0)}:=\beta^{(i,j)},\la^{(1)},\ldots,\la^{(i)},\la^{(i+1)},\ldots,\la^{(i+j)},$$
so that for each intermediate partition, all singletons stay as singletons, and the two parts in the same pair remain consecutive. Consequently, we have that for $1\le k\le i+j$,
\begin{align}\label{weight-preserv}
\ell(\la^{(k)})=\ell(\beta^{(i,j)}),\; \sol(\la^{(k)})=\sol(\beta^{(i,j)}).
\end{align}

\begin{proof}[Proof of Lemma~\ref{lem:main bij}]
We are going to need forward and backward moves for both the singleton and the pair, so in total four kinds of operations are summarized in Table~\ref{tab:4moves} below.

\begin{table}[h!]
\centering
\begin{tabular}{|c|c|c|}
\hline
before & move & after\\
\hline
$a$ & forward & $a+1$ \\
$a$ & backward & $a-1$ \\
$[b,b+1]$ & forward & $[b+1,b+2]$ \\
$[b,b+1]$ & backward & $[b-1,b]$ \\
\hline
\end{tabular}
\medskip
\caption{Four kinds of moves for Lemma~\ref{lem:main bij}.}\label{tab:4moves}
\end{table}

We break the description of $\varphi$ into two phases.
\begin{description}
  \item[Phase I] We use parts from $\mu$ to increase the singletons of $\beta^{(i,j)}$. Denote $s_k:=2j+2k-1$ for $1\le k\le i$. Firstly, we apply $\mu_i$ forward moves to the singleton $s_i$ to arrive at a new partition $\la^{(1)}$, whose largest part is $s_i'=s_i+\mu_i$. Note that $s_i'$ remains a singleton of $\la^{(1)}$. Next, apply $\mu_{i-1}$ forward moves to $s_{i-1}$, turning $\la^{(1)}$ into $\la^{(2)}$ with a new part $s_{i-1}'=s_{i-1}+\mu_{i-1}$. Note that since $\mu_{i-1}\le\mu_i$, we have $s_{i-1}'\le s_i'-2$, so again $s_{i-1}'$ remains a singleton in $\la^{(2)}$. So on and so forth, until we reach the final step of this phase. Namely, we apply $\mu_1$ forward moves to $s_1$, turning $\la^{(i-1)}$ into $\la^{(i)}$, whose smallest singleton is $s_1'=s_1+\mu_1$. Each step can be analogously justified to satisfy condition \eqref{weight-preserv}.

  \item[Phase II] We use parts from $\eta$ to increase the pairs of $\beta^{(i,j)}$, making adjustments along the way if necessary. Note that each forward move on a pair consumes a weight of $2$, which explains why we require $\eta\in\E$. Denote $p_k:=[2k-1,2k]$ for $1\le k\le j$. For $k=j,j-1,\ldots,1$, we apply $\eta_k/2$ forward moves to the pair $p_k$, turning $\la^{(i+j-k)}$ into $\la^{(i+j-k+1)}$. The weight increment is seen to be $|\la^{(i+j-k+1)}|-|\la^{(i+j-k)}|=2\cdot\frac{\eta_k}{2}=\eta_k$. A crucial thing to notice when we compare these moves with those in phase I is that $\eta$ is independent from $\mu$, so it may well be the case that certain pair overtake one or more existing singletons when we perform its forward moves (there are no ``pair overtaking pair'' since $\mu_k\ge \mu_{k-1}$). Such a ``collision'' between pair and singleton causes us trouble, and we fix it with the following ``adjustment''. Here and elsewhere, we indicate the pair or singleton being moved in boldface.
  $$
  \begin{gathered}
  (\text{parts $\leq t-1$}),[ \mathbf{t}, \mathbf{t+1}], t+2, (\text{parts $\geq t+4$}) \\
  \qquad \downarrow \text {one forward move } \\
  (\text{parts $\leq t-1$}),[\mathbf{t+1}, \mathbf{t+2}], t+2, (\text{parts $\geq t+4$}).
  \end{gathered}
  $$
  Clearly the repeated $t+2$ prevents this partition from being strict. We need to make an adjustment as follows.
  $$
  \begin{gathered}
   (\text{parts $\leq t-1$}),[\mathbf{t+1}, \mathbf{t+2}], t+2, (\text{parts $\geq t+4$}) \\
  \qquad \downarrow \text {an adjustment } \\
  (\text{parts $\leq t-1$}),t, [\mathbf{t+2}, \mathbf{t+3}], (\text{parts $\geq t+4$}).
  \end{gathered}
  $$
  After this adjustment there are no repeated parts and the singleton $t+2$ becomes the singleton $t$, since $t+1$ is not a part anymore. Also note that an adjustment does not consume any weights. We need it for keeping the process well-defined. The final partition $\la^{(i+j)}$ is taken to be the image $\varphi(\beta^{(i,j)},\mu,\eta)=\la$. Since we have exhausted the parts of $\mu$ and $\eta$ in those forward moves of Phase I and Phase II respectively, naturally we have the weight match $|\la|=|\beta^{(i,j)}|+|\mu|+|\eta|$, while the other two statistics are preserved thanks to \eqref{weight-preserv}.
\end{description}

Next, we describe the inverse map $\varphi^{-1}$, from any strict partition $\la\in\D_{i,j}$ back to the triple $(\beta^{(i,j)},\mu,\eta)$. Since $\la\in\D_{i,j}$, it must contain $i$ singletons and $j$ pairs. We label them from the smallest to the largest as $p_1,p_2,\ldots,p_j$ (for the pairs) and $s_1,s_2,\ldots,s_i$ (for the singletons). The process to recover the triple from $\la^{(i+j)}:=\la$ can be divided into two phases.
\begin{description}
  \item[Phase II'] We perform backward moves on the pairs so as to recover the base partition, and the weight decrements will be collected as parts of $\eta$. We start with $p_1$. Since it is the smallest pair, if there are any parts smaller than it, they must all be singletons. Suppose $p_1=[t,t+1]$ with $r$ singletons preceding it, then we need $t-1-r$ backward moves on $p_1$ to recover its corresponding pair in $\beta^{(i,j)}$, namely $[1,2]$. The weight decrement is recorded as $\eta_1=2(t-1-r)$, the smallest part of $\eta$, and the new partition we get is denoted as $\la^{(i+j-1)}$. Each time $p_1$ passes by a singleton, a {\it normalization} takes place and ``saves'' us one move, explaining the ``$-r$'' in our counting of moves. In effect, each normalization undoes an adjustment that took place in Phase II of the forward mapping $\varphi$. We illustrate with the following example.
  $$
  \begin{gathered}
  (\text{parts $\leq t-3$}), t-2, [ \mathbf{t}, \mathbf{t+1}], (\text{parts $\geq t+2$}) \\
  \qquad \downarrow \text {one backward move } \\
  (\text{parts $\leq t-3$}), t-2, [\mathbf{t-1}, \mathbf{t}], (\text{parts $\geq t+2$})\\
  \qquad \downarrow \text {a normalization } \\
  (\text{parts $\leq t-3$}), [\mathbf{t-2}, \mathbf{t-1}], t, (\text{parts $\geq t+2$}).
  \end{gathered}
  $$
  Note that the normalization creates no weight changes even though the pair $[t-1,t]$ does effectively advance to the left by $1$. In general for $1< k\le j$, suppose the pair $p_k=[t_k,t_k+1]$ and there are $r_k$ singletons inbetween $p_k$ and $[2k-3,2k-2]$, then we apply $t_k-(2k-1)-r_k$ backward moves on $p_k$, turning $\la^{(i+j-k+1)}$ into $\la^{(i+j-k)}$, taking $r_k$ normalizations along the way, and recording the weight derement as $\eta_k:=2(t_k-(2k-1)-r_k)$. When all pairs are back to their locations in the base partition $\beta^{(i,j)}$, we are done with this phase and obtain $\la^{(i)}$. It should be clear that $\eta:=\eta_1+\eta_2+\cdots+\eta_j$ is indeed a partition in $\E_j$.
  \item[Phase I'] Note that the singletons in $\la^{(i)}$ may be in different positions as they were in $\la$, due to potential normalizations when certain pair passes them by. Rename them as $s_1',\ldots,s_i'$. The idea is clearly to reverse Phase I. So for $k=1,2,\ldots,i$, we apply $s_k'-(2j+2k-1)$ backward moves on the singleton $s_k'$, and record the new partition as $\la^{(i-k)}$, the weight decrement as $\mu_k:=s_k'-(2j+2k-1)$. In the end, $\la^{(0)}=\beta^{(i,j)}$ has been recovered, $\mu:=\mu_1+\cdots+\mu_i\in\P_i$ and we are done.
\end{description}
Seeing that Phase II (resp.~Phase I) and Phase II' (resp.~Phase I') are inverse process of each other, we deduce that $\varphi$ is a bijection.
\end{proof}
An example of applying the bijection $\varphi$ and and its inverse $\varphi^{-1}$ is worth sharing here.
\begin{example}
Given $\beta=\beta^{(2,2)}=[1, 2], [3, 4], 5, 7$, $\mu=1+4$ and $\eta=4+4$. First for Phase I, we use parts $4$ and $1$ from the partition $\mu$ to forward move the singletons $7$ and $5$, respectively. Then we have 
\begin{align*}
\la^{(1)} &=[1, 2], [3, 4], 5, \mathbf{11};\\
\la^{(2)} &=[1, 2], [3, 4], \mathbf{6}, 11.
\end{align*}
Next for Phase II, we perform $\frac{1}{2}\eta_{2}=2$ forward moves on the largest pair $[3, 4]$, making one adjustment along the way.
$$
\begin{gathered}
\la^{(2)}=[1, 2], [3, 4], 6, 11\\
\qquad \downarrow \text {the first forward move on $[3, 4]$} \\
[1, 2], \mathbf{[4, 5]}, 6, 11\\
\qquad \downarrow \text {the second forward move on $[4, 5]$} \\
[1, 2], \mathbf{[5, 6]}, 6, 11\\
\qquad \downarrow \text {an adjustment} \\
[1, 2], 4, \mathbf{[6, 7]}, 11=\la^{(3)}.
\end{gathered}
$$
Then we perform $\frac{1}{2}\eta_{1}=2$ forward moves on the next pair $[1, 2]$, making one adjustment along the way.
$$
\begin{gathered}
\la^{(3)}=[1,2], 4, [6, 7], 11\\
\qquad \downarrow \text {the first forward move on $[1, 2]$}\\
\mathbf{[2, 3]}, 4, [6, 7], 11\\
\qquad \downarrow \text {the second forward move on $[2, 3]$}\\
\mathbf{[3, 4]}, 4, [6, 7], 11\\
\qquad \downarrow \text {an adjustment}\\
2, \mathbf{[4, 5]}, [6, 7], 11=\la^{(4)}=:\la.
\end{gathered}
$$
One verifies that $|\lambda|=35=|\beta|+|\mu|+|\eta|$, $\ell(\la)=\ell(\beta)=6$, and $\sol(\la)=\sol(\beta)=2$. Next we construct $(\beta, \mu, \eta)$ from the partition $\lambda$ via the inverse mapping $\varphi^{-1}$ as follows.
$$
\begin{gathered}
\la^{(4)}:=2, [4, 5], [6, 7], 11\\
\qquad \downarrow \text {one backward move on $[4, 5]$}\\
2, \mathbf{[3, 4]}, [6, 7], 11\\
\qquad \downarrow \text {a normalization}\\
\mathbf{[2, 3]}, 4, [6, 7], 11\\
\qquad \downarrow \text {one backward move on $[2, 3]$}\\
\mathbf{[1, 2]}, 4, [6, 7], 11=\la^{(3)}.
\end{gathered}
$$
This results in $\eta_{1}=2\cdot(4-1-1)=4$.
$$
\begin{gathered}
\la^{(3)}=[1, 2], 4, [6, 7], 11 \\
\qquad \downarrow \text {one backward move on $[6, 7]$}\\
[1, 2], 4, \mathbf{[5, 6]}, 11\\
\qquad \downarrow \text {a normalization}\\
[1, 2], \mathbf{[4, 5]}, 6, 11\\
\qquad \downarrow \text {one backward move on $[4, 5]$}\\
[1, 2], \mathbf{[3, 4]}, 6, 11=\la^{(2)},
\end{gathered}
$$
which means that $\eta_2=2\cdot(6-3-1)=4$. This completes Phase II' and we have $\eta=4+4$. Phase I' is seen to give us $\mu_{1}=6-5=1$, $\mu_{2}=11-7=4$, and $\la^{(0)}=\beta=[1, 2], [3, 4], 5, 7$. Together we arrive at the desired triple $(\beta,\mu,\eta)$.
\end{example}

\section{A combinatorial proof of Theorem~\ref{thm:three (1,2)}}\label{sec:3id}

The combinatorial framework and Theorem~\ref{thm:main} enable us to give alternative proofs of the three identities in Theorem~\ref{thm:three (1,2)}. We begin with the proofs of \eqref{id:(1,2)-2} and \eqref{id:(1,2)-3}.

Clearly, the right hand side of \eqref{id:(1,2)-2}, namely $(-uq;q)_{\infty}$, generates strict partitions with the power of $u$ keeping track of the number of parts, so it suffices to make a change of variables ($x\to 1,~y\to u$) to deduce \eqref{id:(1,2)-2} from \eqref{gf:l-sol}.

Along the same lines, by first setting $u\to uq^3$, then $q^2\to q$, we get the following equivalent form of \eqref{id:(1,2)-3}:
\begin{align}\label{id:(1,2)-3'}
\sum_{i, j\geq 0}\frac{u^{i+2j}q^{i^2+2ij+2j^2+3j}}{(q; q)_i(q^2; q^2)_j}=(1+uq+uq^2)(-uq^3; q)_{\infty}.
\end{align}
Now we can give a bijective proof of \eqref{id:(1,2)-3'} which is analogous to the proof of Theorem~\ref{thm:main}, with additional restrictions on the appearances of parts $1$ and $2$. Alternatively, we observe in retrospect that \eqref{id:(1,2)-3'} is actually equivalent to \eqref{id:(1,2)-2}. Namely, we first denote the left hand side and the right hand side of \eqref{id:(1,2)-2} by $L(u)$ and $R(u)$, respectively, then notice that the right hand side of \eqref{id:(1,2)-3'} can be rewritten as
 \begin{align*}
 (1+uq+uq^2)(-uq^3; q)_{\infty} &= (-uq;q)_{\infty}-u^2q^3(-uq^3;q)_{\infty}=R(u)-u^2q^3R(uq^2).
 \end{align*}
 Simple calculation verifies that the left hand side of \eqref{id:(1,2)-3'}
equals $L(u)-u^2q^3L(uq^2)$. This means that \eqref{id:(1,2)-2} implies \eqref{id:(1,2)-3'}. Conversely, turning the difference equation
\begin{align*}
X(u)-u^2q^3X(uq^2)=(1+uq+uq^2)(-uq^3;q)_{\infty}
\end{align*}
with initial condition $X(0)=1$ into a homogeneous difference equation, we can apply the uniqueness of solution (see~\cite[Lemma~1]{and68}) to deduce \eqref{id:(1,2)-2} from \eqref{id:(1,2)-3'}. So these two identities are indeed equivalent.



Next, we proceed to the proof of Eq.~\eqref{id:(1,2)-1}. The factor $(-1)^j$ appearing in the summand of the left hand side of \eqref{id:(1,2)-1} indicates that we need to construct a ``killing involution'' that explains the massive cancellation from this side, with the fixed points being generated by the right hand side of \eqref{id:(1,2)-1}. For $n\in \N$, let $\O\D(n)$ be the set of strict partitions of $n$ with odd parts only and denote $\O\D=\bigcup_{n\geq 0}\O\D(n)$. For $i, j\geq 0$, let $\C_{i, j}$ be the set of partitions into $i+2j$ parts such that 
\begin{enumerate}[(1)]
    \item the number of occurrences of each part is at most $2$,
    \item the number of repeated parts is $j$, and
    \item the difference between two adjacent distinct parts is at least $2$. 
\end{enumerate}
Denote $\C=\bigcup_{i, j\geq 0}\C_{i, j}$ and $\C(n)=\{\la\in \C: \la\vdash n\}$. Notice that $\O\D\subset\C$. Also, it is not hard to see that
\begin{align}\label{gf:OD}
\sum_{\la\in\O\D}u^{\ell(\la)}q^{\abs{\la}}=(-uq;q^2)_\infty,
\end{align}
which is the right hand side of \eqref{id:(1,2)-1}. Meanwhile, we have the following interpretation concerning the left hand side of \eqref{id:(1,2)-1}.
\begin{lemma}
For $i,j\ge 0$, we have
\begin{align}\label{gf:C_ij}
\sum_{\la\in\C_{i,j}}v^{\ell_r(\la)}u^{\ell_d(\la)}q^{\abs{\la}} &=\frac{v^{j}u^{i+j}q^{i^2+2ij+2j^2}}{(q;q)_{i}(q^2;q^2)_{j}}.
\end{align}
\end{lemma}
\begin{proof}
The proof relies on a b$+$i-decomposition and for the most part parallels the proof of Lemma~\ref{lem:main bij}, so we will point out the main distinctions and be brief on other details. First, all partitions in $\C_{i,j}$ have $j$ repeated parts and $i+j$ different part sizes, so they are weighted uniformly by $v^ju^{i+j}$ in the summation. The next thing to realize is that the base partition has now become
\begin{align*}
\beta^{(i, j)}:=[1, 1], [3, 3], ..., [(2j-1), (2j-1)], (2j+1), (2j+3), ..., (2j+2i-1),
\end{align*}
which could be readily verified to have the correct size $|\beta^{(i, j)}|=i^2+2ij+2j^2$. With the same partitions $\mu\in\P_i$ and $\eta\in\E_j$ producing the increments, we can analogously define forward and backward moves, as well as adjustment and normalization, so as to construct a weight-preserving bijection between $\set{\beta^{(i,j)}}\times\P_i\times\E_j$ and $\C_{i,j}$. An example with one forward move and an adjustment is as follows.
  $$
  \begin{gathered}
  (\text{parts $\leq t-3$}), [ \mathbf{t-1}, \mathbf{t-1}], t+1, (\text{parts $\geq t+3$}) \\
  \qquad \downarrow \text {one forward move } \\
  (\text{parts $\leq t-3$}), [\mathbf{t}, \mathbf{t}], t+1, (\text{parts $\geq t+3$})\\
  \qquad \downarrow \text {an adjustment} \\
  (\text{parts $\leq t-3$}), t-1, [\mathbf{t+1}, \mathbf{t+1}], (\text{parts $\geq t+3$}).
  \end{gathered}
  $$
\end{proof}
Plugging in $v=-1$ in \eqref{gf:C_ij} recovers the summand of the left hand side of \eqref{id:(1,2)-1} and interprets it as a signed and weighted counting of partitions in $\C$. Combining this new insight with \eqref{gf:OD}, we see that Eq.~\eqref{id:(1,2)-1} is equivalent to the following partition theorem.

\begin{theorem}\label{thm:Lovejoy}
For any given $n\ge m\ge 0$, let $\O\D(m,n)$ (resp.~$\C(m,n)$) be the set of partitions in $\O\D(n)$ (resp.~$\C(n)$) with $m$ parts (resp.~$m$ distinct parts), then we have
\begin{align*}
|\set{\la\in\C(m,n): \ell_r(\la) \text{ is even}}|-|\set{\la\in\C(m,n): \ell_r(\la) \text{ is odd}}|=|\O\D(m,n)|.
\end{align*}
\end{theorem}

\begin{remark}
It should be pointed out that in \cite[Thm.~1.9]{lov06}, Lovejoy gave essentially the same partition theorem as Theorem~\ref{thm:Lovejoy}, but his way of proving it was to apply the ``constant term method'', while our approach below is purely combinatorial, reminiscent of Franklin's famed involutive proof of Euler's pentagonal number theorem (cf.~\cite[Thm.~1.6]{andtp}).
\end{remark}



\begin{theorem}\label{thm:Franklin invo}
There exists an involution
\begin{align*}
\psi: \C &\to \C\\
\la &\mapsto \gam,
\end{align*}
such that $\gamma=\la$ if and only if $\la\in\O\D$, while for $\la\in\C\setminus \O\D$, we have $|\la|=|\gam|$, $\ell_d(\la)=\ell_d(\gamma)$, and $\ell_r(\la)\not\equiv\ell_r(\gamma)\pmod 2$. Consequently, Theorem~\ref{thm:Lovejoy} holds true.
\end{theorem}

Before getting into the proof, we need to introduce several concepts for the convenience of our construction of the involution $\psi$. A maximal string of consecutive even singletons (non-repeated parts) contained in a partition $\la\in \C$ is called an {\it even run} of $\lambda$. For example,
$$
\lambda=[2, 2], 5, 8, [11, 11], 13, 16, 18, 20, [22,22], 30, 32, 35
$$
has three even runs, namely, $\{8\}, \{16 ,18, 20\}$ and $\{30, 32\}$. We define 
\begin{align*}
\lrp(\la) &:=\text{the largest repeated part in $\la$},\\
\sle(\la) &:=\text{the smallest part in the largest even run of $\la$}.
\end{align*}
So for our running example $\la$ above, $\lrp(\la)=22$ and $\sle(\la)=30$. We agree that if there is no repeated parts or no even parts in $\la$, then take $\lrp(\la)=-\infty$ or $\sle(\la)=-\infty$, respectively. Notice that if $\la$ has neither repeated parts nor odd parts, then it must be that $\la\in \O\D$.

\begin{proof}[Proof of Theorem \ref{thm:Franklin invo}]
Given a partition $\la\in\C$, if $\lrp(\la)=\sle(\la)=-\infty$, or equivalently $\la\in\O\D$, then we take it to be a fixed point of $\psi$. Otherwise, denote $a=\lrp(\la)$ and $b=\sle(\la)$. If $a=-\infty$, then simply take $\hat{a}:=a$ and go to Case II. In what follows we shall assume that $a>-\infty$. Let us take a closer look at the portion of $\la$ that begins with $[a,a]$:
$$
\la=(\text{parts}\leq a-2), [a, a], a_{1}, a_{2}, \cdots, a_{k}, (\text{parts}\geq a_{k}+2)\in\C,
$$  
where $k$ is the smallest integer such that $a_{k}\geq 2(a+k-1)+2$. If there is no such a part $a_{k}$ then we consider it to be $+\infty$ and take $a_{k-1}$ to be the last part of $\la$. We carry out $(k-1)$ forward moves on $[a, a]$, and one backward move on each of $a_1,a_2,\ldots,a_{k-1}$ to balance out the weight. As a result we get 
$$
\la'=(\text{parts}\leq a-2), (a_{1}-2), \cdots, (a_{k-1}-2), [(a+k-1), (a+k-1)], a_{k}, (\text{parts}\geq a_{k}+2).
$$
Let $\hat{a}=a+k-1$, then comparing the values of $2\hat{a}$ and $b=\sle(\la)$, we have the following two cases to consider.
\begin{description}

\item[Case I] If $2\hat{a}\geq b-2$, then we replace $[\hat{a}, \hat{a}]$ by $2\hat{a}$ in $\lambda'$ to get $\gam$. That is,
$$
\gam=(\text{parts}\leq a-2), (a_{1}-2), \cdots, (a_{k-1}-2), 2\hat{a}, a_{k}, (\text{parts}\geq a_{k}+2).
$$
Our choice of $k$ ensures that $a_{k-1}\leq 2(a+k-2)+1=2(a+k-1)-1=2\hat{a}-1$, then $a_{k-1}-2\leq 2\hat{a}-3$. On the other hand, $a_{k}\geq 2(a+k-1)+2=2\hat{a}+2$. Hence $\gam\in\C$ is well-defined. One checks that $\abs{\gam}=\abs{\la}$, $\ell_d(\gam)=\ell_d(\la)$, and $\ell_r(\gam)=\ell_r(\la)-1$. 

Moreover, we claim that $2\hat{a}=\sle(\gam)$. Indeed, for the case $2\hat{a}=b-2$ we see $a_{k}=b$, thus $2\hat{a}$ precedes and replaces $b$ to be the new smallest part in the last even run. While for the case $a_k-2\ge 2\hat{a}>b-2$, so $b\le a_{k-1}$, we know from previous discussion that $a_{k-1}-2\leq 2\hat{a}-3$. It means that in $\gam$, the new part $2\hat{a}$ is separated from the even run that originally was led by $b$, making $\{2\hat{a}\}$ effectively the new largest even run of $\gam$. In either case we have $2\hat{a}=\sle(\gam)$. In addition, make a note that $\gam$ now belongs to Case II. 

\item[Case II] If $2\hat{a}< b-2$, that is, $2\hat{a}\leq b-4$ since $b$ is an even number, then we write out further parts of $\la$ and $\la'$ that sit between $a_k$ and $b$ (if any), as follows.
\begin{align*}
\lambda=(\text{parts}\leq a-2), [a, a], a_{1}, a_{2}, \cdots, a_{k}, \cdots, a_{k+m}, b, (\text{parts}\geq b+2),\\
\lambda'=(\text{parts}\leq a-2), (a_{1}-2), \cdots, (a_{k-1}-2), [\hat{a}, \hat{a}], a_{k},\cdots, a_{k+m}, b, (\text{parts}\geq b+2).
\end{align*}
Make the following operations on $\la$. First compare $b$ with $a_{k+m}$, if $\frac{b}{2}-1\leq a_{k+m}$, then make one backward move on $b$ and one forward move on $a_{k+m}$. Next compare $b-2$ with $a_{k+m-1}$ (see if $\frac{b-2}{2}-1\le a_{k+m-1}$) and make moves as needed. So on and so forth until we encounter the smallest $t$ satisfying $a_{k+m-t}\leq\frac{b-2t}{2}-2$, and we denote this intermediate partition as
$$
\gam'=(\cdots), [a, a], a_{1}, \cdots, a_{k+m-t}, b-2t, (a_{k+m-t+1}+2), \cdots, (a_{k+m}+2), (\cdots).
$$
Now we replace part $b-2t$ with the pair $[\frac{b}{2}-t, \frac{b}{2}-t]$ to get $\gam$. If such a $t$ does exist, then we see that $\gam$ is well-defined since $a_{k+m-t}\leq \frac{b}{2}-t-2$ and $a_{k+m-t+1}\geq \frac{b}{2}-t$. One verifies that $\lrp(\gam'')=\frac{b}{2}-t$. On the other hand, if there exists no such $t$, that is, the chain of intermediate partitions stops at
\begin{align*}
\cdots, [a, a], b-2k-2m, a_{1}+2, \cdots,
\end{align*}
we make the split $b-2k-2m\to[\frac{b}{2}-k-m,\frac{b}{2}-k-m]$ to get $\gam$ and need to justify that $\gam$ is still a valid partition in $\C$. Indeed, since $b$ is the smallest part of largest even run in $\la$, we have
$a_{k}\le a_{k+m}-2m\le b-3-2m$. Also note that $a_{k}\geq 2\hat{a}+2=2a+2k$, so we know that $2a+2k\leq b-3-2m$, that is, $a\leq \frac{b}{2}-k-m-2$ since $b$ is even. In both cases, it is straightforward to check that $\abs{\gam}=\abs{\la}$, $\ell_d(\gam)=\ell_d(\la)$, and $\ell_r(\gam)=\ell_r(\la)+1$. It helps to also realize that $\lrp(\gam)=\frac{b}{2}-t$ (or $\lrp(\gam)=\frac{b}{2}-k-m$ in the case that $t$ does not exist) and $\gam$ belongs to Case I.
\end{description}



With everything that have been discussed above, we conclude that the map $\psi$ is indeed an involution that fixes $\O\D$, while for $\la\in\C\setminus\O\D$, its image $\gam=\psi(\la)$ satisfies
$$|\la|=|\gam|,~\ell_d(\la)=\ell_d(\gamma), \text{ and } \ell_r(\la)\not\equiv\ell_r(\gamma)\pmod 2.$$

\end{proof}

Some examples shall serve us well at this point. We offer one example with a step-by-step breaking down of applying the involution $\psi$ on a certain partition $\la$, and another example showing the complete correspondences for the case of $n=18$.

\begin{example}
Suppose that 
$$
\la=[2, 2], 4, 6, 9, [11, 11], 14, 16, 18.
$$
One sees that $b=\sle(\la)=14$ and $a=\lrp(\la)=11$. We need to make several forward moves on $[a,a]$ to get $\hat{a}$ for $\la$ as follows
\begin{align*}
\la=[2, 2], 4, 6, 9, [11, 11], 14, 16, 18 \\
\qquad \downarrow \text { a forward move since $11\leq 14\leq 23$ } \\
[2, 2], 4, 6, 9, 12, [12, 12], 16, 18\\
\qquad \downarrow \text { a forward move since $12\leq 16\leq 25$ } \\
[2, 2], 4, 6, 9, 12, 14, [13, 13], 18\\
\qquad \downarrow \text { a forward move since $13\leq 18\leq 27$ } \\
[2, 2], 4, 6, 9, 12, 14, 16, [14, 14]. 
\end{align*}
Then we know that $\hat{a}=14$. Since $2\hat{a}=28\geq b-2=12$, we see that $\la$ belongs to Case I so we get 
$$
\gam=[2, 2], 4, 6, 9, 12, 14, 16, 28.
$$
Conversely, it also reqires some forward moves to get $\hat{a}$ for $\gam$:
\begin{align*}
\gam=[2, 2], 4, 6, 9, 12, 14, 16, 28 \\
\qquad \downarrow \text { a forward move since $2\leq 4\leq 5$ } \\
2, [3, 3], 6, 9, 12, 14, 16, 28\\
\qquad \downarrow \text { a forward move since $3\leq 6\leq 7$ } \\
2, 4, [4, 4], 9, 12, 14, 16, 28\\
\qquad \downarrow \text { a forward move since $4\leq 9\leq 9$ } \\
2, 4, 7, [5, 5], 12, 14, 16, 28. 
\end{align*}
So we have that $b'=\sle(\gam)=28$ and $\hat{a'}=5$, then $2\hat{a'}=10\leq b-4=24$, meaning that $\gam$ belongs to Case II. We derive 

\begin{align*}
\gam=[2, 2], 4, 6, 9, 12, 14, 16, 28 \\
\qquad \downarrow \text { a forward move since $13\leq 16\leq 28$ } \\
[2, 2], 4, 6, 9, 12, 14, 26, 18\\
\qquad \downarrow \text { a forward move since $12\leq 14\leq 26$ } \\
[2, 2], 4, 6, 9, 12, 24, 16, 18\\
\qquad \downarrow \text { a forward move since $11\leq 12\leq 24$ } \\
[2, 2], 4, 6, 9, 22, 14, 16, 18.\\
\qquad \downarrow \text { no more moves since $9\le 22/2-2=9$ }
\end{align*}
Hence we get $\la=[2, 2], 4, 6, 9, [11, 11], 14, 16, 18
$ by splitting $22$ into $[11, 11]$.
\end{example}

\begin{example}
For $n=18$, one enumerates the fixed points to be $1+17,3+15,5+13, 7+11$, and $1+3+5+9$. The remaining partitions in $\C(18)$ are paired up according to our involution $\psi$ as follows.
\\

\begin{minipage}[c]{0.5\textwidth}
\centering
\begin{tabular}{r|l}
18 & [9+9]\\
2+16 & 2+[8+8]\\
4+14 & 4+[7+7]\\
6+12 & [5+5]+8\\
8+10 & [4+4]+10\\
1+3+14 & 1+3+[7+7]\\
1+4+13 & [1+1]+3+13\\
1+5+12 & 1+[5+5]+7\\
1+6+11 & 1+[3+3]+11\\
1+7+10 & 1+[4+4]+9\\
2+4+12 & 2+4+[6+6]\\
2+5+11 & [1+1]+5+11\\
2+6+10 & 2+[4+4]+8
\end{tabular}
\end{minipage}
\begin{minipage}[c]{0.5\textwidth}
\centering
\begin{tabular}{r|l}
2+7+9 & [1+1]+7+9\\
3+5+10 & [3+3]+5+7\\
3+6+9 & [2+2]+5+9\\
4+6+8 & [2+2]+6+8\\
1+3+6+8 & [1+1]+3+5+8\\
\text{[1+1]+[3+3]+10} & [1+1]+[3+3]+[5+5]\\
\text{[1+1]+[4+4]+8} & [1+1]+6+10\\
\ [2+2]+[4+4]+6 & [2+2]+4+10\\
\ [1+1]+[8+8] & [1+1]+16\\
\ [1+1]+4+[6+6] & [1+1]+4+12\\
\ [2+2]+[7+7] & [2+2]+14\\
\ [3+3]+[6+6] & [3+3]+12
\end{tabular}
\end{minipage}
\end{example}

As alluded to in \cite{CW23}, double sum identities like \eqref{id:(1,2)-1} could as well be established by summing over one of the index first and utilizing $q$-series manipulations. We supply such an analytic proof here for the sake of completeness. Recall the $q$-binomial theorem~\cite[p.~354, (II.3)]{GR04}
$$
\sum\limits_{n\geq 0}\frac{(a; q)_{n}}{(q; q)_{n}}u^{n}=\frac{(au; q)_{\infty}}{(u; q)_{\infty}},\ |u|<1,
$$
and two of its corollaries
\begin{equation}
\sum\limits_{n\geq 0}\frac{u^n}{(q; q)_{n}}=\frac{1}{(u; q)_{\infty}},\quad \sum\limits_{n\geq 0}\frac{q^{\binom{n}{2}}u^n}{(q; q)_{n}}=(-u; q)_{\infty},\ |u|<1.\label{eq:bino cor}
\end{equation}
We also need the $q$-Chu-Vandermonde identity in reverse order of summation~\cite[p.~354, (II.7)]{GR04}:
\begin{align}\label{id:qChuVan}
\sum\limits_{i=0}^{n}\frac{(a; q)_{i}(q^{-n}; q)_{i}}{(q; q)_{i}(c; q)_{i}}\left(\frac{cq^{n}}{a}\right)^{i}=\frac{(c/a; q)_{n}}{(c; q)_{n}}.
\end{align}
Let $c=-q, a\rightarrow +\infty$, we have that
$$
\begin{aligned}
\lim\limits_{a\rightarrow +\infty}(a; q)_{i}\left(\frac{cq^{n}}{a}\right)^{i}&=\lim\limits_{a\rightarrow +\infty}(\frac{1}{a}-1)(\frac{1}{a}-q)\cdots(\frac{1}{a}-q^{i-1})(-q^{n+1})^{i}\\
&=(-1)^{i}q^{\binom{i}{2}}(-q^{n+1})^{i}\\
&=q^{\binom{i+1}{2}+ni}.
\end{aligned}
$$
Putting this back to \eqref{id:qChuVan} we get
\begin{equation}
\sum\limits_{i=0}^{n}\frac{(q^{-n}; q)_{i}q^{\binom{i+1}{2}+ni}}{(q; q)_{i}(-q; q)_{i}}=\frac{1}{(-q; q)_{n}}.\label{eq:new1}
\end{equation}

\begin{proof}[Analytic proof of equation \eqref{id:(1,2)-1}]
We have that 
\begin{equation}
\begin{aligned}
\sum\limits_{i, j\geq 0}\frac{q^{i^2+2ij+2j^2}x^{i+2j}y^{j}}{(q; q)_{i}(q^2; q^2)_{j}}&=\sum\limits_{i, j\geq 0}\frac{q^{(i+j)^{2}}x^{i+j}}{(q; q)_{i+j}}\sum\limits_{j\geq 0}\frac{(q^{i+1}; q)_{j}q^{j^2}(xy)^{j}}{(q; q)_{j}(-q; q)_{j}}\\
&\overset{\text{set $n=i+j$}}{=}\sum\limits_{n\geq 0}\frac{q^{n^2}x^n}{(q; q)_{n}}\sum\limits_{j=0}^{n}\frac{(q^{-n}; q)_{j}(-xy)^{j}q^{nj+\binom{j+1}{2}}}{(q; q)_{j}(-q; q)_{j}}.\label{eq:new2}
\end{aligned}
\end{equation}
The last step is because
$$
\begin{aligned}
(q^{n-j+1}; q)_{j}q^{j^2}&=(1-q^{n-j+1})(1-q^{n-j+2})\cdots(1-q^n)q^{j^2}\\
&=(-q^{n-j+1})(-q^{n-j+2})\cdots (-q^n)(1-q^{-n+j-1})\cdots(1-q^{-n})q^{j^2}\\
&=(-1)^{j}(q^{-n}; q)_{j}q^{\binom{j+1}{2}+nj}.
\end{aligned}
$$
If $xy=-1$, say $x=u,~y=-u^{-1}$, then 
$$
\text{LHS of \eqref{eq:new2}}=\sum\limits_{i, j\geq 0}\frac{(-1)^{j}u^{i+j}q^{i^2+2ij+2j^2}}{(q; q)_{i}(q^2; q^2)_{j}}=\text{LHS of \eqref{id:(1,2)-1}},
$$
and
$$
\begin{aligned}
\text{RHS of \eqref{eq:new2}}&=\sum\limits_{n\geq 0}\frac{q^{n^2}u^n}{(q; q)_{n}}\sum\limits_{j=0}^{n}\frac{(q^{-n}; q)_{j}q^{nj+\binom{j+1}{2}}}{(q; q)_{j}(-q; q)_{j}} \overset{\text{by \eqref{eq:new1}}}{=}\sum\limits_{n\geq 0}\frac{q^{n^2}u^n}{(q; q)_{n}(-q; q)_{n}}\\
&=\sum\limits_{n\geq 0}\frac{q^{n^2}u^n}{(q^2; q^2)_{n}}\overset{\text{by \eqref{eq:bino cor}}}{=}(-uq; q^2)_{\infty}=\text{RHS of \eqref{id:(1,2)-1}}.
\end{aligned}
$$
So we see that \eqref{eq:new2} specializes to \eqref{id:(1,2)-1} and the proof is complete. 
\end{proof}

\section{A combinatorial proof of Theorem~\ref{thm:Wei-2para}}\label{sec:Wei}

In this section, we aim to construct an involution to prove Theorem \ref{thm:Wei-2para}. Comparing the left hand side of \eqref{id:Wei-2para} with the right hand side of \eqref{gf:l-sol} prompts us to make the following changes of variables. Let $x\rightarrow xq$ and $y\rightarrow yq$ in \eqref{id:Wei-2para}, we have
\begin{equation}
\sum\limits_{i, j\geq 0}\frac{q^{i^2+2ij+2j^2+j}}{(q; q)_{i}(q^2; q^2)_{j}}x^{i}y^{2j}=(yq; q)_{\infty}\sum\limits_{j\geq 0}\frac{(-x/y; q)_{j}}{(q; q)_{j}(yq; q)_{j}}q^{\binom{j+1}{2}}y^{j}.\label{id:m-Wei-2para}
\end{equation}
Linking the left hand side of \eqref{id:m-Wei-2para} to the double sum expression for $D^{\sol,\ell}(x,y;q)$ as given by \eqref{gf:l-sol}, we see that
\begin{align*}
\sum_{\la\in \D}x^{\sol(\la)}y^{\ell(\la)-\sol(\la)}q^{|\la|}=\sum\limits_{i, j\geq 0}\frac{q^{i^2+2ij+2j^2+j}}{(q; q)_{i}(q^2; q^2)_{j}}x^{i}y^{2j}.
\end{align*}
In effect, the left hand side of \eqref{id:m-Wei-2para} can now be viewed as the weighted generating function for all strict partitions: $\sum_{\la\in\D}w(\la)q^{\abs{\la}}$, where the weight $w(\la):=w(\la_1)w(\la_2)\cdots w(\la_m)$, with the weight on each part $\la_i$ given by
\begin{align*}
w(\la_i)=
\begin{cases}
x & \text{if $\la_i$ is the largest part in a sequence of odd length,}\\
y & \text{otherwise.}
\end{cases}
\end{align*}


Meanwhile, the right hand side of \eqref{id:m-Wei-2para} can be rewritten as
\begin{align*}
\text{RHS}=\sum\limits_{j\geq 0}\frac{(y+x)(y+xq)\cdots (y+xq^{j-1})q^{\binom{j+1}{2}}}{(q; q)_{j}}\cdot(yq^{j+1}; q)_{\infty}=\sum_{j\geq 0}A_{j}\cdot B_{j}.
\end{align*}
According to this decomposition, we can interpret the summand $A_jB_j$ as the generating function of the following set of strict weighted partition pairs $(\la,\mu)\in\A_j\times\B_j$.
\begin{itemize}
\item Let $\A_j$ be the set of weighted strict partitions $\la$ with $j$ distinct parts, such that the parts are labeled as either $x$ or $y$, and a part $\la_i$ can be labeled as $x$ only when $\la_{i+1}-\la_i\geq 2$. We make the convention $\la_{j+1}=+\infty$, so that the largest part $\la_j$ can be labeled as either $x$ or $y$. Let the weight $w_A(\la)$ be the product of the labels of all the parts of $\la$. We see that $A_j=\sum_{\la\in\A_j}w_A(\la)q^{\abs{\la}}$.
\item Let $\B_j$ be the set of weighted strict partitions $\mu$ with each part no less than $j+1$, and each part is labeled as $-y$. Let the weight $w_B(\mu)$ be the product of the labels of all the parts of $\mu$. We have that $B_j=\sum_{\mu\in\B_j}w_B(\mu)q^{\abs{\mu}}$.
\end{itemize}

We are going to construct an involution $\theta$ on $\bigcup_{j\ge 0}\A_j\times \B_j$. Before that, we introduce several useful notions.
\begin{Def}
A sequence in a given weighted partition in $\A_j$ is called {\it bad}, if either it has even length and its last part is labeled as $x$, or it has odd length and its last part is labeled as $y$. A sequence that is not bad is called a {\it good} sequence.
\end{Def}
For example, the partition $\la=1_y+2_y+3_y+5_x+8_y+9_x\in\A_6$ has three sequences, among which the first $(1_y,2_y,3_y)$ and the third $(8_y,9_x)$ are the bad ones. The following is a key definition that is best understood in terms of the Ferrers diagram of partitions.

\begin{Def}
Given a partition $\la=\la_{1}+\la_{2}+\cdots+\la_{\ell}$, the {\it $k$-th L-shape} for $1\leq k\leq \ell$ refers to the shaded portion of its Ferrers diagram $[\la]$ as shown below.
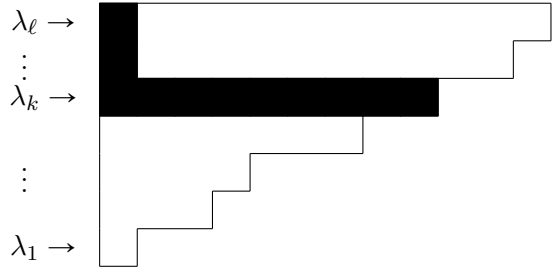
\begin{ferrers}
	    \addsketchrows{12+11+9+7+4+3+1}
        \addline{0.5}{0}{0.5}{-1.5}
        \addline{0}{-1.5}{3.5}{-1.5}
        \addline{0}{-1}{4.5}{-1}
        \highlightcellbycolor{1}{1}{black}
        \highlightcellbycolor{2}{1}{black}
        \highlightcellbycolor{3}{1}{black}
        \highlightcellbycolor{3}{2}{black}
           \highlightcellbycolor{3}{3}{black}
        \highlightcellbycolor{3}{4}{black}
        \highlightcellbycolor{3}{5}{black}
        \highlightcellbycolor{3}{6}{black}
                \highlightcellbycolor{3}{7}{black}
                        \highlightcellbycolor{3}{8}{black}
        \highlightcellbycolor{3}{9}{black}
        \addtext{-0.75}{-0.25}{$\lambda_{\ell}\rightarrow$}
        \addtext{-1}{-0.75}{$\vdots$}
\addtext{-0.75}{-1.25}{$\lambda_{k}\rightarrow$}
 \addtext{-1}{-2.25}{$\vdots$}
\addtext{-0.75}{-3.25}{$\lambda_{1}\rightarrow$}
\end{ferrers}
\end{Def}

We denote the size of the $k$-th L-shape as $s_{k}=s_{\la,k}=\la_{k}+\ell-k$. It is worth noting that $s_1,s_2,\ldots,s_{\ell(\la)}$ is a weakly increasing sequence, and $s_a=s_b$ if and only if $\la_a$ and $\la_b$ belong to the same sequence of $\la$.

\begin{Def}
Given a weighted strict partition $\la\in\bigcup_{j\ge 0}\A_j$, we introduce a new statistic $\sfb(\la)$, the size of the L-shape corresponding to the first bad sequence of $\la$. More precisely, if $\la$ has no bad sequences, we let $\sfb(\la):=+\infty$. Otherwise suppose its first bad sequence begins with the part $\la_k$, then we let $\sfb(\la):=s_k$. 
\end{Def}
Now we are ready to construct the involution $\theta$.

\begin{lemma}
There exists an involution 
\begin{align*}
\theta: \bigcup_{j\ge 0}(\A_{j}\times \B_j) &\to \bigcup_{j\ge 0}(\A_{j}\times \B_j)\\
(\la, \mu) &\mapsto (\beta, \gam),
\end{align*}
such that $|\la|+|\mu|=|\beta|+|\gam|$. And when $(\la,\mu)\neq(\beta,\gam)$, we have $w_A(\la)w_B(\mu)=-w_A(\beta)w_B(\gam)$. Consequently, \eqref{id:m-Wei-2para} holds true and Theorem~\ref{thm:Wei-2para} follows.
\end{lemma}

\begin{proof}
Given a pair of weighted strict partitions $(\la, \mu)\in \A_{j}\times \B_j$, we will compare the values $\sfb(\la)$ and $\mu_1$ (set $\mu_1:=+\infty$ when $\mu$ is the empty partition). If $\sfb(\la)=\mu_1=+\infty$, i.e., $\la$ has no bad sequences and $\mu$ is empty, then we take $(\beta,\gam):=(\la,\mu)$. These pairs are all of the fixed points of $\theta$. Clearly, when $\la$ has no bad sequences, we have $w_A(\la)=w(\la)$. Conversely, each strict partition $\la$ with weight $w(\la)$ can be viewed as a weighted partition in $\A_{\ell(\la)}$ without bad sequences and $w_A(\la)=w(\la)$. So we see that
\begin{align*}
\sum_{(\la,\mu)\in\bigcup_{j\ge 0}\A_j\times\B_j,~\theta(\la,\mu)=(\la,\mu)}w_A(\la)w_B(\mu)q^{\abs{\la}+\abs{\mu}} &=\sum_{\la\in\D}w(\la)q^{\abs{\la}},
\end{align*}
precisely the left hand side of \eqref{id:m-Wei-2para}. It suffices to show that the remaining pairs cancel out each other completely in their weights. Suppose the first bad sequence of $\la$ begins at the part $\la_k$ and recall that $\sfb(\la)=\la_k+\ell(\la)-k$. There are two cases to consider.


\begin{description}
\item[Case I] If $\sfb(\la)<\mu_1$ (including the case that $\mu_1=+\infty$), then we delete the $k$-th L-shape from $\la$ to get a new partition $\beta$, i.e.,
$$
\beta=\beta_1+\cdots+\beta_{j-1}, \text{ where }
\beta_i:=\begin{cases}
\la_i & 1\le i\le k-1,\\
\la_{i+1}-1 & k\le i\le j-1.
\end{cases}
$$
And all parts of $\beta$ receive the same labels as their counterparts in $\la$. A moment of reflection reveals that $\beta$ as defined above is still properly $w_A$-weighted, so $\beta\in\A_{j-1}$. Note that by definition, $\la_k$ must have label $y$ in $\la$, hence $w_A(\la)=yw_A(\beta)$. To get $\gam$, we append $\sfb(\la)$ as a new smallest part to $\mu$, i.e.,
$$
\gam=\gam_1+\cdots+\gam_{\ell(\mu)+1}, \text{ where }
\gam_i:=\begin{cases}
\sfb(\la) & i=1,\\
\mu_{i-1} & 2\le i\le \ell(\mu)+1.
\end{cases}
$$
We see that $\gam$ is a strict partition in $\B_{j-1}$, since $\sfb(\la)\ge s_1=\la_1+j-1\ge j$. We also have $(-y)w_B(\mu)=w_B(\gam)$. To sum up, we have shown that $(\beta, \gam)\in \A_{j-1}\times \B_{j-1}$, $|\la|+|\mu|=|\beta|+|\gam|$, and $w_A(\la)w_B(\mu)=-w_A(\beta)w_B(\gam)$ as desired. 

Moreover, we claim that $\sfb(\beta)\ge\gam_{1}$, so that the image pair $(\beta,\gam)$ is in case II below. Indeed, deleting the $k$-th L-shape from $\la$ and keeping the labels of the remaining parts will turn the bad sequence led by $\la_k$ into a good one (or delete it completely when $\la_k$ is a stand-alone sequence itself), and keep the other sequences intact (i.e., good ones stay good, bad ones stay bad). If $\beta$ has no bad sequences, then $\sfb(\beta)=+\infty$ and our claim holds trivially. If $\beta$ does have bad sequence(s), the first of which must be induced from a bad sequence in $\la$ that immediately follows the bad sequence containing $\la_k$. In other words, there must be a certain $l>k$, such that $\sfb(\beta)=s_l-1\ge s_k=\sfb(\la)=\gam_1$. 


\item[Case II] If $\sfb(\la)\ge \mu_1$ (including the case that $\sfb(\la)=+\infty$), then we remove $\mu_1$ from $\mu$ to get a new partition $\gam\in\B_{j+1}$, and clearly $w_B(\mu)=(-y)w_B(\gam)$.

Next, we describe a unique way to insert $\mu_{1}$ into $\la$ as an L-shape to get our target partition $\beta$. Find the smallest index $l$ such that $s_l=\la_l+j-l\ge \mu_1-1$. Note that $s_k=\sfb(\la)\ge \mu_1$ so such an $l$ must exist\footnote{For the particular case that $\sfb(\la)=+\infty$ and $\la_j<\mu_1-1$, simply append $\mu_1$ to the end of $\la$ as a new largest part with label $y$. The new partition is taken to be $\beta$.}. Now insert $\mu_1$ into $\la$ as the $l$-th L-shape, i.e., we let
\begin{align*}
\beta=\beta_1+\cdots+\beta_{j+1}, \text{ where }
\beta_i:=\begin{cases}
\la_i & 1\le i\le l-1,\\
\mu_1-j+l-1 & i=l,\\
\la_{i-1}+1 & l+1\le i\le j+1.
\end{cases}
\end{align*}
Furthermore, $\beta_i$ receive the same label as $\la_i$ for $1\le i\le l-1$, as $\la_{i-1}$ for $l+1\le i\le j+1$, and the new part $\beta_l$ is labeled as $y$. We trust the reader to verify the following facts.
\begin{enumerate}
  \item $\beta_{l-1}+2\le\beta_l\le\beta_{l+1}-1$ and $\beta$ is indeed a properly $w_A$-weighted strict partition, i.e., $\beta\in\A_{j+1}$.
  \item $w_A(\beta)=yw_A(\la)$.
  \item $\beta_l$ begins a bad sequence in $\beta$ and $\sfb(\beta)=\mu_1<\mu_2=\gam_1$.
\end{enumerate}

These facts are sufficient to yield that $\abs{\la}+\abs{\mu}=\abs{\beta}+\abs{\gam}$, $w_A(\la)w_B(\mu)=-w_A(\beta)w_B(\gam)$, and $(\beta,\gam)\in\A_{j+1}\times\B_{j+1}$ is a pair in case I.

\end{description}

Finally, our constructions in the above two cases are clearly inverse of each other, making $\theta$ an involution with the claimed properties.


\end{proof}


We conclude this section with two examples, one shows the process of applying the involution $\theta$ in both cases, the other displays all correspondences under $\theta$ for partition pairs $(\la,\mu)$ with a fixed value of $\abs{\la}+\abs{\mu}$. For simpler notation, in these examples we assume that parts in $\la$ without a subscript are labeled by $y$, and parts in $\mu$ without a subscript are labeled by $-y$.

\begin{example}
Given $\la=(1, 2, 4, 5, 6, 7_{x}, 9, 10, 11, 12_{x})\in\A_{10}$ and $\mu=(11, 12)\in\B_{10}$, we first note that $(4,5,6,7_x)$ is the first bad sequence in $\la$, so
$$
\sfb(\la)=s_{3}=7+4=11\geq \mu_{1}=11,
$$
and we are in case II. We remove $\mu_1=11$ from $\mu$ to get $\gam=(12)\in\B_{11}$. Next, we realize that the smallest index $l=1$ since $s_1=\la_1+10-1=10\ge \mu_1-1=10$. Thus, we insert $11$ as the first L-shape into $\la$, to derive
$$
\beta=(1, 2, 3, 5, 6, 7, 8_{x}, 10, 11, 12, 13_{x})\in\A_{11}.
$$
Inversely, noting that the first sequence $(1,2,3)$ is already bad in $\beta$, we have $\sfb(\beta)=s_1=1+10=11<\gam_1=12$, which means we are in case I. Hence we delete the first L-shape from $\beta$ and insert $11$ as a new smallest part into $\gam$, rendering the original pair
$$
\la=(1, 2, 4, 5, 6, 7_{x}, 9, 10, 11, 12_{x})\text{ and }\mu=(11, 12).
$$
\end{example}

\begin{example}
Among all partition pairs $(\la,\mu)\in\bigcup_{j\ge 0}(\A_j\times\B_j)$ satisfying $|\la|+|\mu|=6$, there are four fixed points under the involution $\theta$. Namely, $(6_{x}, \emptyset), (2_x+4_x, \emptyset), (1_x+5_x, \emptyset)$, and $(1+2+3_x, \emptyset)$. The remaining ones are paired up via $\theta$ as follows.
\\

\begin{minipage}[c]{0.5\textwidth}
\centering
\begin{tabular}{r|l}
case I & case II\\
$(6, \emptyset)$ & $(\emptyset, 6)$\\
$(1, 2+3)$ & $(\emptyset, 1+2+3)$ \\
$(1, 5)$ & $(\emptyset, 1+5)$\\
$(1_x+5, \emptyset)$ & $(1_x, 5)$\\
$(2, 4)$ & $(\emptyset, 2+4)$\\
$(2_x+4, \emptyset)$ & $(2_x, 4)$\\
\end{tabular}
\end{minipage}
\begin{minipage}[c]{0.5\textwidth}
\centering
\begin{tabular}{r|l}
case I & case II\\
$(2+4, \emptyset)$ & $(3, 3)$\\
$(2+4_x, \emptyset)$ & $(3_x, 3)$\\
$(1+5, \emptyset)$ & $(4, 2)$\\
$(1+5_x, \emptyset)$ & $(4_x, 2)$\\
$(1+2+3, \emptyset)$ & $(1+2, 3)$\\
$(1+2_x, 3)$ & $(1_x, 2+3)$\\
\end{tabular}
\end{minipage}
\end{example}

\section{Partition theoretical interpretations of Andrews-Uncu identities}\label{sec:Andrews-Uncu}

To exploit our combinatorial approach further, we demonstrate in this section how to apply our b$+$i-decomposition to interpret the series sides of \eqref{id:AU-1} and \eqref{id:AU-2}. In fact, we were motivated to go a bit more general and introduce the following notion of $(k,a)$-strict partition.
\begin{Def}
For $k>a\ge 1$, a strict partition is called a {\it $(k,a)$-strict partition} if all of its sequences are of length either $a$ or $k$ modulo $k$. The set of all $(k,a)$-strict partitions is denoted as
\begin{align*}
\D_{k,a}=\{\la\in \D: \text{the length of $\la$'s sequence}\equiv 0, a~(\mod\ k)\}.
\end{align*}
For a given $\la\in\D_{k,a}$, denote $\sl(\la)=\sl_{k,a}(\la)$ the number of sequences in $\la$ whose lengths are congruent to $a$ modulo $k$. We also introduce the following generating function for $(k,a)$-strict partitions.
\begin{align*}
D^{\sl, \ell}_{k,a}(x, y; q):=\sum_{\la\in \D_{k,a}}x^{\sl_{k,a}(\la)}y^{\ell(\la)}q^{|\la|}.
\end{align*}
\end{Def}

Note that $\D_{2,1}=\D$, $\sl_{2,1}(\la)=\sol(\la)$ for any $\la\in\D$, and $D_{2,1}^{\sl,\ell}(x,y;q)=D^{\sol,\ell}(x,y;q)$. Hence the following theorem reduces to Theorem~\ref{thm:main} in the case of $(k,a)=(2,1)$. 

\begin{theorem}\label{thm:main-ka}
We have
\begin{align*}
D^{\sl, \ell}_{k,a}(x, y; q)=\sum_{i, j\geq 0}\frac{x^{i}y^{ai+kj}q^{i^2\binom{a+1}{2}+akij+\binom{kj+1}{2}}}{(q^a; q^a)_i(q^k; q^k)_j}.
\end{align*}
\end{theorem}

Let $\T^m$ be the set of partitions $\la $ into multiples of $m$, and $\T^m_n=\{\la\in \T^m: \ell(\la)\leq n\}$ for any $m\in \N$. Denote 
$$
\D_{k, a; i, j}=\{\la\in \D_{k, a}: \sl(\la)=i, \ell(\la)=ai+kj\}.
$$
Similarly, we will use the b+i-decomposition to construct a bijection between the set of triples $(\beta^{(k, a; i, j)}, \mu, \eta)$ and the set $\D_{k, a; i, j}$, where the base partition
\begin{align*}
\beta^{(k, a; i, j)}&=[1, 2,\ldots, k], [k+1, k+2,\ldots, 2k], \cdots, [(j-1)k+1, (j-1)k+2,\ldots, jk], \\
&(jk+1,\ldots, jk+a), (jk+a+2,\ldots, jk+2a+1), \cdots, (jk+(i-1)a+i,\ldots, jk+ia+i-1),
\end{align*}
which is the partition in $\D_{k,a;i,j}$ with the smallest possible weight, and $\mu\in \T^a_i$, $\eta\in \T^k_j$ for $i, j\geq 0$. 

Now we can extend Lemma~\ref{lem:main bij} to the following $(k,a)$-strict partition case, which can be then utilized to establish Theorem~\ref{thm:main-ka} bijectively.

\begin{lemma}\label{lem:amodk bij}
For fixed $k>a\ge 1$, and any given $i, j\geq 0$, there exists a bijection 
\begin{align*}
\varphi_{k, a}=\varphi_{k, a; i, j}: \set{\beta^{(k, a; i, j)}}\times\T^a_i\times\T^k_j &\to \D_{k, a; i, j}\\
(\beta^{(k, a; i, j)},\mu,\eta) &\mapsto \la,
\end{align*}
such that $|\la|=|\beta^{(k, a; i, j)}|+|\mu|+|\eta|$, $\ell(\la)=\ell(\beta^{(k, a; i, j)})$, and $\sl(\la)=\sl(\beta^{(k, a; i, j)})$.
\end{lemma}

The proof of Lemma~\ref{lem:amodk bij} is for the most part similar to that of Lemma \ref{lem:main bij}. To avoid repetition, we just point out the differences between the two proofs and provide two examples showing the correspondences. Firstly, the singletons have become sequences of length $a$, and $\mu$ is a partition into multiples of $a$. Therefore, an increment on a sequence of length $a$ incurred by a part, say $ma$, of $\mu$ results in $m$ forward moves with every move adding $1$ to each of the $a$ parts in that sequence. Secondly, the pairs have become sequences of length $k$, and operations on them decided by $\nu$ need to be modified accordingly. To be precise, we write out all four operations explicitly for easier comparison.
\begin{description}
\item[The forward move and adjustment for Lemma \ref{lem:amodk bij}]
$$
\begin{gathered}
(\text { parts } \leq t-1),[ \mathbf{t},...,  \mathbf{t+k-1}], (t+k, t+k+1, ..., t+k+a-1), (\text { parts } \geq t+k+a+1) \\
\qquad \downarrow \text { one forward move } \\
(\text { parts } \leq t-1),[\mathbf{t+1},... \mathbf{t+k}], (t+k, t+k+1, ..., t+k+a-1), (\text { parts } \geq t+k+a+1)\\
\qquad \downarrow \text { adjustment } \\
(\text { parts } \leq t-1), (t, t+1, ..., t+a-1), [\mathbf{t+1+a}, ..., \mathbf{t+k+a}], (\text { parts } \geq t+k+a+1).
\end{gathered}
$$

\item[The backward move and normalization for Lemma \ref{lem:amodk bij}]
$$
\begin{gathered}
(\text { parts } \leq t-1), (t, t+1, ..., t+a-1), [\mathbf{t+1+a}, ..., \mathbf{t+k+a}], (\text { parts } \geq t+k+a+1)\\
\qquad \downarrow \text { one backward move } \\
(\text { parts } \leq t-1), (t, t+1, ..., t+a-1), [\mathbf{t+a}, ..., \mathbf{t+k+a-1}], (\text { parts } \geq t+k+a+1)\\
\qquad \downarrow \text { normalization } \\
(\text { parts } \leq t-1),[ \mathbf{t},...,  \mathbf{t+k-1}], (t+k, t+k+1, ..., t+k+a-1), (\text { parts } \geq t+k+a+1) .
\end{gathered}
$$
\end{description}

\begin{example}
The following correspondences are via $\varphi_{3,1;2,2}$ and $\varphi_{3,2;2;2}$, respectively.
\begin{align*}
&\left(\beta^{(3, 1; 2, 2)}=[1, 2, 3], [4, 5, 6], 7, 9,~\mu=1+2,~\eta=3+6\right)\longleftrightarrow \lambda=(2, 3, 4, 5, 7, 8, 9, 11).\\
&\left(\beta^{(3, 2; 2, 2)}=[1, 2, 3], [4, 5, 6], (7, 8), (10, 11), \mu=2+4, \eta=3+9\right) \\
&\qquad\qquad\qquad\qquad \longleftrightarrow \lambda=(2, 3, 4, 5, 6, 9, 10, 11, 12, 13).
\end{align*}
\end{example}

Next, as a direct application of Theorem~\ref{thm:main-ka}, we see that $D^{\sl,\ell}_{3,1}(-1,-1;q)$ matches the series side of \eqref{id:AU-1}. In terms of partition theorem, we have the following result that is equivalent to \eqref{id:AU-1}. A Franklin-type involutive proof of it is highly desired.
\begin{corollary}\label{cor:AU1 par}
For any $n\ge 1$, the excess of the number of $(3,1)$-strict partitions $\la$ of $n$ with $\ell(\la)+\sl(\la)$ being even over the number of $(3,1)$-strict partitions $\la$ of $n$ with $\ell(\la)+\sl(\la)$ being odd, equals the number of partitions of $n$ into parts congruent to $1$ modulo $3$.
\end{corollary}

\begin{example}
There are seven partitions contained in $\D_{3,1}(10)$, among which six have an even value for $\ell(\la)+\sl(\la)$, namely, $10$, $1+9$, $2+8$, $3+7$, $4+6$, and $1+3+6$, while the remaining partition $\mu=1+2+3+4$ has $\la(\mu)+\sl(\mu)=4+1=5$, an odd number. On the other hand, there are $5=6-1$ partitions of $10$ into parts congruent to $1$ modulo $3$. They are $10$, $1^3+7$, $1^2+4^2$, $1^6+4$, and $1^{10}$.
\end{example}

In the same vein, one sees that $D_{3,1}^{\sl,\ell}(-q^{2/3},-q^{1/3};q)$ agrees with the series side of \eqref{id:AU-2}. In order to state a partition theoretical counterpart of \eqref{id:AU-2}, some extra efforts are need. We first change the base partition to the following one. For $i,j\ge 0$, let
\begin{align*}
\beta^{(i, j)}=[2, 2, 3], [5, 5, 6], \cdots, [(3j-1), (3j-1), 3j], (3j+2), (3j+4), \cdots, (3j+2i).
\end{align*}
Comparing this with $\beta^{(3,1;i,j)}$, we see that the first part in each triple $[t,t+1,t+2]$ as well as each singleton $(s)$ have been increased by $1$. This can be achieved by making changes of variables in $D_{3,1}^{\sl,\ell}(x,y;q)$: $x\to xq^{2/3}$, $y\to yq^{1/3}$ (the negative signs will be taken account of by certain signed counting; see Corollary~\ref{cor:AU2 par}). It then brings us new gap conditions between consecutive parts. Namely, we let $\W_{i, j}$ be the set of partitions $\la$ into $i+3j$ parts satisfying that (cf.~the definition of $\C_{i,j}$)
\begin{enumerate}[(1)]
    \item the number of occurrences of each part is at most $2$,
    \item the smallest part is at least $2$,
    \item $\ell_r(\la)=j$, and
    \item for $1\leq k< i+3j$, we have 
      \begin{align*}
      \la_{k+1}-\la_{k}
      \begin{cases}
      = 1 & \text{ if $\la_{k}-\la_{k-1}=0$,}\\
      \geq 2 & \text{ if $\la_k-\la_{k-1}=1$,}\\
      \text{$\ge 2$ or $=0$} & \text{ if $\la_k-\la_{k-1}\ge 2$,}
      \end{cases}
      \end{align*} 
      where we set $\la_0=0$ as a convention.
\end{enumerate}

A bijection between the set of triples $(\beta^{(i, j)}, \mu, \eta)$ and the set $\W_{i, j}$ can be similarly constructed from the following moves.

\begin{description}
\item[The forward move and adjustment]
$$
\begin{gathered}
(\text { parts } \leq t-3),[ \mathbf{t-1}, \mathbf{t-1}, \mathbf{t}], t+2, (\text { parts } \geq t+4) \\
\qquad \downarrow \text { one forward move } \\
(\text { parts } \leq t-3),[\mathbf{t}, \mathbf{t}, \mathbf{t+1}], t+2, (\text { parts } \geq t+4)\\
\qquad \downarrow \text { adjustment } \\
(\text { parts } \leq t-3), t-1, [\mathbf{t+1}, \mathbf{t+1}, \mathbf{t+2}], (\text { parts } \geq t+4).
\end{gathered}
$$

\item[The backward move and normalization]
$$
\begin{gathered}
(\text { parts } \leq t-3), t-1, [ \mathbf{t+1}, \mathbf{t+1}, \mathbf{t+2}], (\text { parts } \geq t+4) \\
\qquad \downarrow \text { one backward move } \\
(\text { parts } \leq t-3), t-1, [\mathbf{t}, \mathbf{t}, \mathbf{t+1}], (\text { parts } \geq t+4)\\
\qquad \downarrow \text { normalization } \\
(\text { parts } \leq t-3), [\mathbf{t-1}, \mathbf{t-1}, \mathbf{t}], t+2, (\text { parts } \geq t+4).
\end{gathered}
$$
\end{description}
 
\begin{lemma}\label{lem:inva of main}
For $i, j\geq 0$, there exists a bijection 
\begin{align*}
\varphi'_{3,1}=\varphi'_{3,1;i,j}: \set{\beta^{(i,j)}}\times\P_i\times\T^3_j &\to \W_{i,j}\\
(\beta^{(i,j)},\mu,\eta) &\mapsto \la,
\end{align*}
such that $|\la|=|\beta^{(i,j)}|+|\mu|+|\eta|$, $\ell(\la)=\ell(\beta^{(i,j)})$, and $\ell_r(\la)=\ell_r(\beta^{(i,j)})$.
\end{lemma} 

Denote $\W:=\bigcup_{i,j\ge 0}\W_{i,j}$ and $\W(n):=\set{\la\in\W: \la\vdash n}$. We immediately deduce from Lemma~\ref{lem:inva of main} the following double sum generating function.
\begin{align}\label{gf:W-xy}
\sum_{\la\in\W}x^{\ell(\la)-3\ell_r(\la)}y^{\ell(\la)}q^{\abs{\la}}=D_{3,1}^{\sl,\ell}(xq^{2/3},yq^{1/3};q)=\sum_{i,j\ge 0}\frac{x^{i}y^{i+3j}q^{\frac{3j(3j+1)}{2}+i^2+3ij+i+j}}{(q;q)_i(q^3;q^3)_j}.
\end{align}

Setting $x=y=-1$ in \eqref{gf:W-xy}, we derive the following partition theorem that is equivalent to \eqref{id:AU-2}. 
\begin{corollary}\label{cor:AU2 par}
For any $n\ge 1$, the excess of the number of permutations in $\W(n)$ with an even number of repeated parts over the number of permutations in $\W(n)$ with an odd number of repeated parts, equals the number of partitions of $n$ into parts congruent to $2$ or $3$ modulo $6$.
\end{corollary}

\begin{example}
There are five partitions contained in $\W(10)$, among which four have zero repeated parts, namely, $10$, $2+8$, $3+7$, and $4+6$, while the remaining partition $3+3+4$ has one repeated part. On the other hand, there are $3=4-1$ partitions of $10$ into parts congruent to $2$ or $3$ modulo $6$. They are $2+8$, $2^2+3^2$, and $2^5$.
\end{example}

Viewing the statements of Corollaries \ref{cor:AU1 par} and \ref{cor:AU2 par}, it is tempting to ask for direct Franklin-type involutive proofs of them. Such proofs have eluded us so far and are highly desired.

\section{Conclusion and future works}

Intrigued by a bunch of double sum Rogers-Ramanujan type identities derived by Cao-Wang~\cite{CW23}, Wang-Wang~\cite{WW23}, Wei-Yu-Ruan~\cite{WYR23}, Andrews-Uncu~\cite{AU23}, Chern~\cite{che23}, and Wang~\cite{wan23} respectively using various methods, we introduce in this work a new partition statistic $\sol(\la)$, the number of sequences of odd length in a strict partition $\la$. We also develop a b$+$i-decomposition for strict partitions that preserves the statistic $\sol$. This combinatorial framework enables us to derive partition theoretical interpretations to all of the aforementioned identities. In most cases we manage to devise Franklin-type involutions to prove the identity combinatorially. 

Recall that Euler's odd--distinct partition theorem states that for any non-negative integer $n$, the set of strict partitions of $n$ and the set of partitions of $n$ into odd parts are equinumerous. Our new statistic $\sol$ refines the counting of strict partitions, it is natural to ask if there exists a corresponding statistic for partitions into odd parts, so as to give a refinement of Euler's theorem. We will address this question in a forthcoming paper~\cite{FL24}.
 
Another line of potential future works~\cite{FL25} focuses on Rogers-Ramanujan partitions. These are restricted partitions that arise in MacMahon's partition theoretical interpretation of the series sides of Rogers-Ramanujan identities (see Theorem~\ref{thm:RR-partition}). We shall consider a statistic analogous to $\sol$ and use it to refine the generating function of Rogers-Ramanujan partitions. This approach leads us to new combinatorial proofs of several multi-sum identities originally derived by Wang~\cite{wan22} and Li-Wang~\cite{LW23} via other methods.

\section*{Acknowledgement}
Both authors were partially supported by the National Natural Science Foundation of China grant 12171059 and the Mathematical Research Center of Chongqing University.

\end{document}